\documentclass[10pt]{article}
\usepackage[utf8]{inputenc}
\usepackage{lmodern}
\usepackage[T1]{fontenc}
\usepackage{amsmath}
\usepackage{amsfonts}
\usepackage{amssymb}
\usepackage{amsthm}
\usepackage{mathrsfs}
\usepackage{stmaryrd}
\usepackage{mathtools}
\usepackage{graphicx}
\usepackage[left=2.9cm,right=2.9cm,top=3cm,bottom=3cm]{geometry}
\usepackage{enumitem}
\usepackage{hyperref}
\usepackage{tikz}
\usepackage{dsfont}
\usepackage{multicol}
\usepackage{xcolor}

\usepackage[maxbibnames=99]{biblatex}
\renewcommand{\epsilon}{\varepsilon}

\newcommand{\constplusone}[1]{\the\numexpr\value{#1}+1\relax}
\newcommand{\constplustwo}[1]{\the\numexpr\value{#1}+2\relax}
\newcommand{\constminusone}[1]{\the\numexpr\value{#1}-1\relax}
\newcommand{\constminustwo}[1]{\the\numexpr\value{#1}-2\relax}

\theoremstyle{plain}
\newtheorem{thm}{Theorem}
\newtheorem{lem}[thm]{Lemma}

\theoremstyle{definition}

\theoremstyle{remark}
\newtheorem{rem}[thm]{Remark}

\begin{document}

\begin{center}
\large{\textbf{\uppercase{Uniform estimates for solutions of nonlinear focusing damped wave equations}}}

\vspace{\baselineskip}

\large{\textsc{Thomas Perrin}}

\vspace{\baselineskip}

\large{\today}

\vspace{\baselineskip}
\end{center}

\noindent
\textbf{Abstract.} For a damped wave (or Klein-Gordon) equation on a bounded domain, with a focusing power-like nonlinearity satisfying some growth conditions, we prove that a global solution is bounded in the energy space, uniformly in time. Our result applies in particular to the case of a cubic equation on a bounded domain of dimension $3$.

\section*{Introduction}

\paragraph{Setting and main result.} Let $\Omega$ be a compact Riemannian manifold of dimension $d \geq 3$, with or without boundary. Let $\beta \in \mathbb{R}$ be such that the Poincaré inequality
\begin{equation}\label{eq_poincare}
\int_\Omega \left(\left\vert \nabla u \right\vert^2 + \beta \vert u \vert^2 \right) \mathrm{d}x \gtrsim \int_\Omega \vert u \vert^2 \mathrm{d}x, \quad u \in H_0^1(\Omega),
\end{equation}
is satisfied. For real-valued initial data $\left( u^0, u^1 \right) \in H_0^1(\Omega) \times L^2(\Omega)$, $\gamma \in L^\infty(\Omega, \mathbb{R}_+)$, and $f \in \mathscr{C}^1(\mathbb{R}, \mathbb{R})$, consider the equation
\[\left \{
\begin{array}{rcccl}\label{KG_nL}\tag{$\ast$}
\square u + \gamma \partial_t u + \beta u & = & f(u) & \quad & \text{in } \mathbb{R} \times \Omega, \\
(u(0), \partial_t u(0)) & = & \left( u^0, u^1 \right) & \quad & \text{in } \Omega, \\
u & = & 0 & \quad & \text{on } \mathbb{R} \times \partial \Omega.
\end{array}
\right.\]
Set $F(s) = \int_0^s f(\tau) \mathrm{d}\tau$. We assume that $f(0) = 0$, and that there exist $1 < p < + \infty$, $1 < p_0 < + \infty$, $C > 0$, $\epsilon > 0$ such that
\begin{equation}\label{hyp_f_1}
\left\vert f(s_1) - f(s_2) \right\vert \leq C \vert s_1 - s_2 \vert \left( 1 + \vert s_1 \vert^{p - 1} + \vert s_2 \vert^{p - 1} \right), \quad s_1, s_2 \in \mathbb{R},
\end{equation}
\begin{equation}\label{hyp_f_2}
s f(s) \geq (2 + \epsilon) F(s), \quad s \in \mathbb{R},
\end{equation}
and
\begin{equation}\label{hyp_f_3}
\left\Vert u \right\Vert_{L^2}^{p_0 + 1} \leq C \int_\Omega F(u) \mathrm{d}x, \quad u \in H_0^1(\Omega).
\end{equation}
The fact that $\Omega$ is bounded will only be used to ensure the existence of a nonlinearity satisfying (\ref{hyp_f_1}), (\ref{hyp_f_2}) and (\ref{hyp_f_3}). We will always assume that $\Omega$ and $f$ are such that the Cauchy problem (\ref{KG_nL}) can be solved locally (see Theorem \ref{thm_existence_nL_waves} for a precise statement). 

For $\left( u^0, u^1 \right) \in H_0^1(\Omega) \times L^2(\Omega)$, write $\left\Vert u^0 \right\Vert_{H_0^1}^2 = \int_\Omega \left(\left\vert \nabla u^0 \right\vert^2 + \beta \vert u^0 \vert^2 \right) \mathrm{d}x$ and set 
\[E\left( u^0, u^1 \right) = \frac{1}{2} \left\Vert u^0 \right\Vert_{H_0^1}^2 + \frac{1}{2} \left\Vert u^1 \right\Vert_{L^2}^2 - \int_\Omega F(u^0) \mathrm{d}x.\]
For $\left( u^0, u^1 \right) \in H_0^1(\Omega) \times L^2(\Omega)$, if the solution $u$ of (\ref{KG_nL}) exists on some interval $[0, T]$, then one has the energy equality
\begin{equation}\label{eq_energy_equality}
E\left( u(t_1), \partial_t u(t_1) \right) = E\left( u(t_0), \partial_t u(t_0) \right) - \int_{t_0}^{t_1} \int_\Omega \gamma \left\vert \partial_t u \right\vert^2 \mathrm{d}t \mathrm{d}x, \quad 0 \leq t_0 \leq t_1 \leq T.
\end{equation}

Our main result is the following.

\begin{thm}\label{thm_main_damped}
Assume that $f$ satisfies \textnormal{(\ref{hyp_f_1})}, with $p < \frac{d + 2}{d - 2}$, \textnormal{(\ref{hyp_f_2})}, and \textnormal{(\ref{hyp_f_3})}.
\begin{enumerate}[label=(\roman*)]
\item There exists $C = C(f, \gamma) > 0$ such that for all $\left( u^0, u^1 \right) \in H_0^1(\Omega) \times L^2(\Omega)$, if the solution $u$ of \textnormal{(\ref{KG_nL})} exists on $\mathbb{R}_+$, then
\[ E\left( u^0, u^1 \right) \geq E\left( u(t), \partial_t u(t) \right) \geq - C, \quad t \geq 0.\]
\item There exist $c_0 = c_0(f, \gamma) > 0$, $c_1 = c_1(f, \gamma) > 0$, and $c_2 = c_2(f) > 0$ such that for all $\left( u^0, u^1 \right) \in H_0^1(\Omega) \times L^2(\Omega)$, if the solution $u$ of \textnormal{(\ref{KG_nL})} exists on $\mathbb{R}_+$, then
\begin{equation}\label{eq_thm_L2_estimate}
\left\Vert u(t) \right\Vert_{L^2}^2 \leq \left\Vert u^0 \right\Vert_{L^2}^2 e^{- c_2 t} + \left( c_0 + c_1 \left\vert E\left( u^0, u^1 \right) \right\vert \right) \left( 1 - e^{- c_2 t} \right), \quad t \geq 0.
\end{equation}
\item Assume, in addition, that $f$ satisfies \textnormal{(\ref{hyp_f_1})} with $p \leq \frac{d}{d - 2}$. Then, there exists $c = c(f, \gamma) > 0$ such that the function $\alpha : s \mapsto c \exp(c s)$ satisfies the following property. For all $\left( u^0, u^1 \right) \in H_0^1(\Omega) \times L^2(\Omega)$, if the solution $u$ of \textnormal{(\ref{KG_nL})} exists on $\mathbb{R}_+$, then
\begin{equation}\label{eq_thm_H1_estimate_1}
\left\Vert \left( u(t), \partial_t u(t) \right) \right\Vert_{H_0^1 \times L^2} \leq \alpha \left( \left\Vert u^0 \right\Vert_{H_0^1}^2 + \left\Vert u^1 \right\Vert_{L^2}^2 \right), \quad t \geq 0.
\end{equation}
In addition, there exists $T \geq 0$, which depends on $f$, $\gamma$ and $\left\Vert u^0 \right\Vert_{H_0^1}^2 + \left\Vert u^1 \right\Vert_{L^2}^2$, such that
\begin{equation}\label{eq_thm_H1_estimate_2}
\left\Vert \left( u(t), \partial_t u(t) \right) \right\Vert_{H_0^1 \times L^2} \leq \alpha \left( \left\vert E\left( u^0, u^1 \right) \right\vert \right), \quad t \geq T.
\end{equation}
\end{enumerate}
\end{thm}

All constants in this article may depend on $\Omega$. Note that when we specify that a constant depends on $f$, it depends, in fact, only on the constants appearing in (\ref{hyp_f_1}), (\ref{hyp_f_2}), and (\ref{hyp_f_3}). 

\paragraph{Examples.} A typical example of $f$ satisfying (\ref{hyp_f_1}) for some $p > 1$, (\ref{hyp_f_2}), and (\ref{hyp_f_3}) for some $p_0 > 1$, is given by
\[f(s) = \lambda_1 s \vert s \vert^{\alpha_1 - 1} + \cdots + \lambda_n s \vert s \vert^{\alpha_n - 1}, \quad s \in \mathbb{R},\]
with $n \in \mathbb{N}^\ast$, $\lambda_i > 0$ for all $i \in \llbracket 1, n \rrbracket$, and $p_0 \leq \alpha_1 \leq \cdots \leq \alpha_n \leq p$. Assumption (\ref{hyp_f_3}) is satisfied by such $f$ since $\Omega$ is bounded. As previously mentioned, the fact that $\Omega$ is bounded will not be used elsewhere in this article. Theorem \ref{thm_main_damped}-\emph{(i)} and Theorem \ref{thm_main_damped}-\emph{(ii)} can be applied to such $f$ if $p < \frac{d + 2}{d - 2}$, and Theorem \ref{thm_main_damped}-\emph{(iii)} can also be applied if $p \leq \frac{d}{d - 2}$. For example, Theorem \ref{thm_main_damped} applies to the focusing cubic wave equation on a compact manifold of dimension $3$.

\paragraph{Connection with the existing literature.} The main objective of this article is to extend the results of \cite{Cazenave1985} to the case of a damped equation. In \cite{Cazenave1985}, Cazenave assumes that $f$ satisfies (\ref{hyp_f_1}) with $p = \frac{d}{d - 2}$ if $d \geq 3$. Note that this assumption is not satisfied by all subcritical nonlinearity. The usual critical exponent $\frac{d + 2}{d - 2}$ is the one that gives $F(u) \in L^1(\Omega)$ by the Sobolev embedding $H^1 \hookrightarrow L^{\frac{2d}{d - 2}}$, implying that the energy is finite, whereas (\ref{hyp_f_1}) yields $f(u) \in L^2(\Omega)$. However, some of the results of \cite{Cazenave1985} are valid for $f$ satisfying (\ref{hyp_f_2}) with $p < \frac{d + 2}{d - 2}$ if $d \geq 3$ : more precisely, a careful reading of \cite{Cazenave1985} shows that if $\gamma = 0$, then Theorem \ref{thm_main_damped} holds true without assumption (\ref{hyp_f_3}), and with $\Omega$ bounded or unbounded. Unfortunately, we have not managed to find a proof that is valid for any damping $\gamma$, without the additional assumption (\ref{hyp_f_3}). The author of \cite{Cazenave1985} also considers the cases $d = 1, 2$, with an additional assumption on $f$ when $d = 2$. For simplicity, we have not included the cases $d = 1, 2$, although the methods of this article and of \cite{Cazenave1985} can be adapted to those cases.

In \cite{feireisl}, Feireisl considers the case $\Omega = \mathbb{R}^d$, with a damping equal to a positive constant, and a nonlinearity satisfying the restrictive condition $1 < p < 1 + \min \left( \frac{2}{d - 2}, \frac{4}{d} \right)$. Lemma 4.2 of \cite{feireisl} states that a solution with a bounded energy is global (for positive times), and is bounded in $H^1(\mathbb{R}^d) \times L^2(\mathbb{R}^d)$. However, we have not found in \cite{feireisl} a proof of the fact that a global solution has a bounded energy.

Some related work has been done by Burq, Raugel and Schlag. More precisely, in \cite{BurqRaugelSchlag}, the authors prove that a global solution is bounded in the energy space, in the case $\Omega = \mathbb{R}^d$, with a damping equal to a positive constant, for a general focusing subcritical nonlinearity, but with a radial assumption. The proof of the boundedness of \cite{BurqRaugelSchlag} uses dynamical systems arguments. In \cite{Burq_Raugel_Schlag_weak}, they consider the case of a damping which is constant in space, and tends to zero as $t$ tends to infinity, still with $\Omega = \mathbb{R}^d$ and a radial assumption. They adapt the proof of the $L^2$-bound of Cazenave, using the fact that their damping tends to zero as $t$ tends to infinity. 

\paragraph{Outline of the article.} 
In Section 1, we recall the local Cauchy theory for (\ref{KG_nL}), and the expression of the derivative of the square of the $L^2$-norm of a solution. In Section 2, 3 and 4, we prove respectively the energy bound (Theorem \ref{thm_main_damped}-\emph{(i)}), the $L^2$-bound (Theorem \ref{thm_main_damped}-\emph{(ii)}), and the $H^1$-bound (Theorem \ref{thm_main_damped}-\emph{(iii)}). An appendix contains proofs of some elementary lemmas.

\paragraph{Acknowledgements.} I warmly thank Thomas Duyckaerts and Jérôme Le Rousseau for their constant support and guidance.

\section{Preliminaries}

We recall the local Cauchy theory for (\ref{KG_nL}).

\begin{thm}\label{thm_existence_nL_waves}
Consider $f$ satisfying \textnormal{(\ref{hyp_f_1})} for some $p < \frac{d + 2}{d - 2}$. For any real-valued initial data $\left( u^0, u^1 \right) \in H_0^1(\Omega) \times L^2(\Omega)$, there exist a maximal time of existence $T \in (0, + \infty]$ and a unique solution $u$ of \textnormal{(\ref{KG_nL})} in $\mathscr{C}^0([0, T), H_0^1(\Omega)) \cap \mathscr{C}^1([0, T), L^2(\Omega))$. If $T < + \infty$, then 
\[\left\Vert \left( u(t), \partial_t u(t) \right) \right\Vert_{H_0^1 \times L^2} \underset{t \rightarrow T^-}{\longrightarrow} + \infty.\]
\end{thm}

The proof of Theorem \ref{thm_existence_nL_waves} relies on Strichartz estimates and a fixed point argument (see for example \cite{GinibreVelo}). If $\Omega$ is a manifold without boundary, then Theorem \ref{thm_existence_nL_waves} follows from the Strichartz estimates proved in \cite{KeelTao}. In the case of a manifold with boundary, we \emph{assume} that $f$ is such that Theorem \ref{thm_existence_nL_waves} holds true : we refer to \cite{BlairSmithSogge} for Strichartz estimates for a large range of exponents. Ivanovici's counterexamples in \cite{Ivanovici} show that Strichartz estimates are not true for the full range of exponents in the case of a manifold with boundary. 

\begin{lem}
Consider $u$ a solution of \textnormal{(\ref{KG_nL})} on some interval $[0, T]$, and set $M(t) = \left\Vert u(t) \right\Vert_{L^2}^2$, for $t \in [0, T]$. Then $M \in \mathscr{C}^2([0, T], \mathbb{R})$, and for $ t \in [0, T]$, one has $M^{\prime}(t) = 2 \int_\Omega u(t) \partial_t u(t) \mathrm{d}x$, and
\begin{equation}\label{eq_expression_Mprimeprime}
M^{\prime \prime}(t) = 2 \left\Vert \partial_t u(t) \right\Vert_{L^2}^2 - 2 \left\Vert u(t) \right\Vert_{H_0^1}^2 + 2 \int_\Omega u(t) f(u(t)) \mathrm{d}x - 2 \int_\Omega \gamma u(t) \partial_t u(t) \mathrm{d}x.
\end{equation}
\end{lem}

The proof is standard and we omit it. See for example \cite{Cazenave1985} or \cite{Payne-Sattinger}.

\section{The energy bound}

Here, we prove Theorem \ref{thm_main_damped}-\emph{(i)}. Consider $\left( u^0, u^1 \right) \in H_0^1(\Omega) \times L^2(\Omega)$ such that the solution $u$ of \textnormal{(\ref{KG_nL})} exists on $\mathbb{R}_+$. By (\ref{hyp_f_3}), there exists $C_0 = C_0(f) > 0$ such that 
\begin{equation}\label{eq_proof_energy_bounded_1}
\left\vert 2 \int_\Omega \gamma u(t) \partial_t u(t) \mathrm{d}x \right\vert \leq C_0 \left\Vert \gamma \right\Vert^2_{L^\infty} \left( \int_\Omega F(u(t)) \mathrm{d}x \right)^{\frac{2}{p_0 + 1}} + \frac{\epsilon}{4} \left\Vert \partial_t u(t) \right\Vert_{L^2}^2, \quad t \geq 0,
\end{equation}
where $\epsilon$ is given by (\ref{hyp_f_2}). Using (\ref{hyp_f_2}) and (\ref{eq_proof_energy_bounded_1}) in (\ref{eq_expression_Mprimeprime}), one finds
\begin{align}
M^{\prime \prime}(t) \geq & \left( 2 - \frac{\epsilon}{4} \right) \left\Vert \partial_t u(t) \right\Vert_{L^2}^2 - 2 \left\Vert u(t) \right\Vert_{H_0^1}^2 \nonumber\\
& + 2 ( 2 + \epsilon ) \int_\Omega F(u(t)) \mathrm{d}x - C_0 \left\Vert \gamma \right\Vert^2_{L^\infty} \left( \int_\Omega F(u(t)) \mathrm{d}x \right)^{\frac{2}{p_0 + 1}}, \quad t \geq 0. \label{eq_proof_energy_bounded_1bis}
\end{align}
Write $I_0 = \left( \frac{C_0 \left\Vert \gamma \right\Vert^2_{L^\infty}}{\epsilon} \right)^{\frac{p_0 + 1}{p_0 - 1}} \geq 0$, so that 
\[I \geq I_0 \quad \Longrightarrow \quad 2 ( 2 + \epsilon ) I - C_0 \left\Vert \gamma \right\Vert^2_{L^\infty} I^{\frac{2}{p_0 + 1}} \geq (4 + \epsilon) I.\]
Assume by contradiction that there exists $T_0 \geq 0$ such that $E\left( u(T_0), \partial_t u(T_0) \right) < - I_0$. Then for $t \geq T_0$, one has $\int_\Omega F(u(t)) \mathrm{d}x \geq - E\left( u(t), \partial_t u(t) \right) > I_0$, yielding
\[M^{\prime \prime}(t) \geq \left( 2 - \frac{\epsilon}{4} \right) \left\Vert \partial_t u(t) \right\Vert_{L^2}^2 - 2 \left\Vert u(t) \right\Vert_{H_0^1}^2 + ( 4 + \epsilon ) \int_\Omega F(u(t)) \mathrm{d}x, \quad t \geq T_0.\]
By definition of the energy, this gives
\begin{equation}\label{eq_proof_energy_bounded_2}
M^{\prime \prime}(t) \geq \left( 4 + \frac{\epsilon}{4} \right) \left\Vert \partial_t u(t) \right\Vert_{L^2}^2 + \frac{\epsilon}{2} \left\Vert u(t) \right\Vert_{H_0^1}^2 - ( 4 + \epsilon ) E\left( u(t), \partial_t u(t) \right), \quad t \geq T_0.
\end{equation}
One has $E\left( u(t), \partial_t u(t) \right) \leq 0$ for $t \geq T_0$, implying
\[M^{\prime \prime}(t) \geq \left( 4 + \frac{\epsilon}{4} \right) \left\Vert \partial_t u(t) \right\Vert_{L^2}^2, \quad t \geq T_0.\]
In particular, one obtains 
\[\left( M^\prime(t) \right)^2 \leq \left( 1 + \frac{\epsilon}{16} \right)^{-1} M(t) M^{\prime \prime}(t), \quad t \geq T_0.\]
We use the following lemma.

\begin{lem}\label{lem_basic_explosion}
Consider $T \geq 0$, $M \in \mathscr{C}^2([T, + \infty), \mathbb{R}_+)$ and $0 < \delta < 1$ such that $\left( M^\prime(t) \right)^2 \leq \delta M(t) M^{\prime \prime}(t)$ for $t \geq T$. Then $M$ is non-increasing. In particular, one has $M(t) \leq M(T)$ for $t \geq T$.
\end{lem}

By (\ref{eq_proof_energy_bounded_2}), one has $M^{\prime \prime}(t) \geq - (4 + \epsilon) E\left( u(T_0), \partial_t u(T_0) \right) > 0$ for $t \geq T_0$, a contradiction with Lemma \ref{lem_basic_explosion}. Hence, one has $E\left( u(t), \partial_t u(t) \right) \geq - I_0$ for all $t \geq 0$.

\begin{rem}
If $\gamma = 0$, then $I_0 = 0$ : we recover the fact that a global solution of the undamped equation has a non-negative energy.
\end{rem}

\section[The L2-estimates]{The $L^2(\Omega)$-estimate}

Here, we prove Theorem \ref{thm_main_damped}-\emph{(ii)}. Consider $\left( u^0, u^1 \right) \in H_0^1(\Omega) \times L^2(\Omega)$ such that the solution $u$ of \textnormal{(\ref{KG_nL})} exists on $\mathbb{R}_+$. We split the proof into 2 steps.

\paragraph{Step 1 : an estimate on $M^{\prime \prime}$.} \newcounter{constnum} \setcounter{constnum}{0} 
Set $r_0 = \frac{2}{p_0 + 1} \in (0, 1)$. Using (\ref{eq_proof_energy_bounded_1bis}) and the definition of the energy, one finds 
\begin{align*}
M^{\prime \prime}(t) \geq & \left( 4 + \frac{3 \epsilon}{4} \right) \left\Vert \partial_t u(t) \right\Vert_{L^2}^2 + \epsilon \left\Vert u(t) \right\Vert_{H_0^1}^2 - 2 ( 2 + \epsilon ) E\left( u(t), \partial_t u(t) \right) \\
& - C_\theconstnum \left\Vert \gamma \right\Vert^2_{L^\infty} \left( \frac{1}{2} \left\Vert \partial_t u(t) \right\Vert_{L^2}^2 + \frac{1}{2} \left\Vert u(t) \right\Vert_{H_0^1}^2 - E\left( u(t), \partial_t u(t) \right) \right)^{r_0}, \quad t \geq 0,
\end{align*}
with $C_\theconstnum = C_\theconstnum(f)$. \addtocounter{constnum}{1} There exists $C_\theconstnum = C_\theconstnum(\gamma, f) > 0$ such that
\begin{align*}
M^{\prime \prime}(t) \geq & \left( 4 + \frac{3 \epsilon}{4} \right) \left\Vert \partial_t u(t) \right\Vert_{L^2}^2 + \epsilon \left\Vert u(t) \right\Vert_{H_0^1}^2 - 2 ( 2 + \epsilon ) E\left( u(t), \partial_t u(t) \right) \\
& - \frac{C_\theconstnum}{2} \left\Vert \partial_t u(t) \right\Vert_{L^2}^{2 r_0} - C_\theconstnum \left\vert \frac{1}{2} \left\Vert u(t) \right\Vert_{H_0^1}^2 - E\left( u(t), \partial_t u(t) \right) \right\vert^{r_0}, \quad t \geq 0.
\end{align*}
By Theorem \ref{thm_main_damped}-\emph{(i)}, one has $\sup_{t \geq 0} \left\vert E\left( u(t), \partial_t u(t) \right) \right\vert < + \infty$. \addtocounter{constnum}{1} Hence, there exists 
\[C_\theconstnum = C_\theconstnum \left(\gamma, f, \sup_{t \geq 0} \left\vert E\left( u(t), \partial_t u(t) \right) \right\vert \right) > 0\]
such that
\[\epsilon X^2 - 2 ( 2 + \epsilon ) E\left( u(t), \partial_t u(t) \right) - C_{\constminusone{constnum}} \left\vert \frac{1}{2} X^2 - E\left( u(t), \partial_t u(t) \right) \right\vert^{r_0} \geq \frac{\epsilon}{2} X^2 - C_\theconstnum, \quad t \geq 0, \quad X \geq 0,\]
yielding
\[M^{\prime \prime}(t) \geq \left( 4 + \frac{3 \epsilon}{4} \right) \left\Vert \partial_t u(t) \right\Vert_{L^2}^2 - \frac{C_{\constminusone{constnum}}}{2} \left\Vert \partial_t u(t) \right\Vert_{L^2}^{2 r_0} + \frac{\epsilon}{2} \left\Vert u(t) \right\Vert_{H_0^1}^2 - C_\theconstnum, \quad t \geq 0.\]
Note that $C_\theconstnum$ can be chosen of the form 
\[C_\theconstnum = \tilde{C}_0 + \tilde{C}_1 \sup_{t \geq 0} \left\vert E\left( u(t), \partial_t u(t) \right) \right\vert,\]
with $\tilde{C}_0 = \tilde{C}_0\left( \gamma, f \right)$ and $\tilde{C}_1 = \tilde{C}_1\left( \gamma, f \right)$. \addtocounter{constnum}{1} Finally, there exists $C_\theconstnum = C_\theconstnum \left( \gamma, f \right) > 0$ such that
\[\left( 4 + \frac{3 \epsilon}{4} \right) X^2 - \frac{C_{\constminustwo{constnum}}}{2} X^{2 r_0} \geq \left( 4 + \frac{\epsilon}{2} \right) X^2 - C_\theconstnum, \quad X \geq 0,\]
and this gives 
\[M^{\prime \prime}(t) \geq \left( 4 + \frac{\epsilon}{2} \right) \left\Vert \partial_t u(t) \right\Vert_{L^2}^2 + \frac{\epsilon}{2} \left\Vert u(t) \right\Vert_{H_0^1}^2 - C_{\constminusone{constnum}} - C_\theconstnum, \quad t \geq 0.\]
\addtocounter{constnum}{1} Summarizing up, and using the Poincaré inequality (\ref{eq_poincare}), one obtains
\begin{equation}\label{eq_proof_L2_estimate_2}
M^{\prime \prime}(t) \geq \left( 4 + \frac{\epsilon}{2} \right) \left\Vert \partial_t u(t) \right\Vert_{L^2}^2 + C_\theconstnum \left\Vert u(t) \right\Vert_{L^2}^2 - C_{\constplusone{constnum}}, \quad t \geq 0,
\end{equation}
for some $C_\theconstnum = C_\theconstnum(f) > 0$ and $C_{\constplusone{constnum}} > 0$ of the form
\begin{equation}\label{eq_proof_L2_estimate_2_bis}
C_{\constplusone{constnum}} = \tilde{C}_2 + \tilde{C}_3 \sup_{t \geq 0} \left\vert E\left( u(t), \partial_t u(t) \right) \right\vert,
\end{equation}
with $\tilde{C}_2 = \tilde{C}_2(f, \gamma) > 0$ and $\tilde{C}_3 = \tilde{C}_3(f, \gamma) > 0$.

\paragraph{Step 2 : proof of (\ref{eq_thm_L2_estimate}).}
The proof is based on the following elementary lemma.

\begin{lem}\label{lem_basic_exponential_estimate}
Consider $M_0 \in \mathscr{C}^2(\mathbb{R}_+, \mathbb{R})$ such that there exists $C > 0$ such that $M_0^{\prime \prime}(t) \geq C M_0(t)$, for all $t \geq 0$. Then either $M_0(t) \rightarrow + \infty$ as $t \rightarrow + \infty$, or $M_0(t) \leq M_0(0) e^{- \sqrt{C} t }$, for all $t \geq 0$. 
\end{lem}

Set $M_0(t) = C_\theconstnum M(t) - C_{\constplusone{constnum}}$, $t \geq 0$. By (\ref{eq_proof_L2_estimate_2}), one has $M_0^{\prime \prime}(t) \geq C_\theconstnum M_0(t)$, for all $t \geq 0$. Assume that $M_0(t) \rightarrow + \infty$ as $t \rightarrow + \infty$. Then, there exists $T \geq 0$ such that $M_0(t) \geq 0$, for $t \geq T$, implying
\[M^{\prime \prime}(t) \geq \left( 4 + \frac{\epsilon}{2} \right) \left\Vert \partial_t u(t) \right\Vert_{L^2}^2, \quad t \geq T,\]
by (\ref{eq_proof_L2_estimate_2}). This gives 
\[\left( M^\prime(t) \right)^2 \leq \left( 1 + \frac{\epsilon}{8} \right)^{-1} M(t) M^{\prime \prime}(t), \quad t \geq T.\]
Applying Lemma \ref{lem_basic_explosion}, one finds that $M$ is bounded, a contradiction with $M_0(t) \rightarrow + \infty$. Hence, by Lemma \ref{lem_basic_exponential_estimate}, one obtains $M_0(t) \leq M_0(0) e^{- \sqrt{C_\theconstnum} t }$ for $t \geq 0$, that is,
\begin{equation}\label{eq_proof_L2_estimate_3}
M(t) \leq \frac{C_{\constplusone{constnum}}}{C_\theconstnum} + \left( M(0) - \frac{C_{\constplusone{constnum}}}{C_\theconstnum} \right) e^{- \sqrt{C_\theconstnum} t }, \quad t \geq 0.
\end{equation}
Note that using Theorem \ref{thm_main_damped}-\emph{(i)} and the fact that the energy is non-increasing, one has
\begin{equation}\label{eq_proof_L2_estimate_4}
\sup_{t \geq 0} \left\vert E\left( u(t), \partial_t u(t) \right) \right\vert \leq \max \left( \left\vert E\left( u^0, u^1 \right) \right\vert, C_{\constplustwo{constnum}} \right) \leq \left\vert E\left( u^0, u^1 \right) \right\vert + C_{\constplustwo{constnum}},
\end{equation}
for some $C_{\constplustwo{constnum}} = C_{\constplustwo{constnum}}(f, \gamma) > 0$. By (\ref{eq_proof_L2_estimate_2_bis}), this gives
\[\frac{C_{\constplusone{constnum}}}{C_\theconstnum} \leq c_0 + c_1 \left\vert E\left( u^0, u^1 \right) \right\vert,\]
for some $c_0 = c_0(f, \gamma) > 0$ and $c_1 = c_1(f, \gamma) > 0$. Using this in (\ref{eq_proof_L2_estimate_3}), one obtains (\ref{eq_thm_L2_estimate}).

\section[The H1-estimates]{The $H^1(\Omega)$-estimates}

Here, we prove Theorem \ref{thm_main_damped}-\emph{(iii)}. We first need an estimate on $M^\prime$. Note that it does not require the restrictive exponent $\frac{d}{d - 2}$.

\begin{lem}\label{lem_Mprime_estimate}
Assume that $d \geq 3$, and that $f$ satisfies \textnormal{(\ref{hyp_f_1})}, with $p < \frac{d + 2}{d - 2}$, \textnormal{(\ref{hyp_f_2})}, and \textnormal{(\ref{hyp_f_3})}. Then, there exist $c_0 = c_0(f, \gamma) > 0$, $c_1 = c_1(f, \gamma) > 0$, and $c_2 = c_2(f) > 0$ such that for all $\left( u^0, u^1 \right) \in H_0^1(\Omega) \times L^2(\Omega)$, if the solution $u$ of \textnormal{(\ref{KG_nL})} exists on $\mathbb{R}_+$, then
\begin{equation}\label{eq_lem_Mprime_estimate_1}
M^\prime(0) e^{-c_2 t} - \left( c_0 + c_1 \left\vert E\left( u^0, u^1 \right) \right\vert \right) \left( 1 - e^{-c_2 t} \right) \leq M^\prime(t) \leq c_0 + c_1 \left\vert E\left( u^0, u^1 \right) \right\vert, \quad t \geq 0.
\end{equation}
\end{lem}

\begin{proof}
Using the proof of Theorem \ref{thm_main_damped}-\emph{(ii)} (see (\ref{eq_proof_L2_estimate_2}), (\ref{eq_proof_L2_estimate_2_bis}) and (\ref{eq_proof_L2_estimate_4})), one has 
\begin{equation}\label{eq_proof_lem_Mprime_1}
M^{\prime \prime}(t) \geq \left( 4 + \frac{\epsilon}{2} \right) \left\Vert \partial_t u(t) \right\Vert_{L^2}^2 + C_0 \left\Vert u(t) \right\Vert_{L^2}^2 - C_1, \quad t \geq 0,
\end{equation}
for some $C_0 = C_0(f) > 0$ and $C_1 > 0$ of the form
\begin{equation}\label{eq_proof_lem_Mprime_2}
\tilde{C}_0 + \tilde{C}_1 \left\vert E\left( u^0, u^1 \right) \right\vert,
\end{equation}
with $\tilde{C}_0 = \tilde{C}_0(f, \gamma) > 0$ and $\tilde{C}_1 = \tilde{C}_1(f, \gamma) > 0$. For $\eta > 0$, using (\ref{eq_proof_lem_Mprime_1}), one finds
\begin{align*}
\left\vert M^\prime(t) \right\vert & \leq \eta \left\Vert \partial_t u(t) \right\Vert_{L^2}^2 + \frac{1}{\eta} \left\Vert u(t) \right\Vert_{L^2}^2 \\
& \leq \eta \left( 4 + \frac{\epsilon}{2} \right)^{-1} M^{\prime \prime}(t) + \left( \frac{1}{\eta} - C_0 \eta \left( 4 + \frac{\epsilon}{2} \right)^{-1} \right) \left\Vert u(t) \right\Vert_{L^2}^2 + C_1 \eta \left( 4 + \frac{\epsilon}{2} \right)^{-1}, \quad t \geq 0.
\end{align*}
Choosing $\eta$ such that $\frac{1}{\eta} = C_0 \eta \left( 4 + \frac{\epsilon}{2} \right)^{-1}$, one obtains
\begin{equation}\label{eq_proof_lem_Mprime_3}
\left\vert M^\prime(t) \right\vert \leq C_2 M^{\prime \prime}(t) + C_3, \quad t \geq 0,
\end{equation}
with $C_2 = C_2(f) > 0$ and $C_3 > 0$ of the form (\ref{eq_proof_lem_Mprime_2}). Integrating (\ref{eq_proof_lem_Mprime_3}), one finds
\begin{equation}\label{eq_proof_lem_Mprime_4}
\left( M^\prime(t_0) - C_3 \right) e^{-\frac{t_0}{C_2}} \leq \left( M^\prime(t_1) - C_3 \right) e^{-\frac{t_1}{C_2}}, \quad 0 \leq t_0 \leq t_1,
\end{equation}
and
\begin{equation}\label{eq_proof_lem_Mprime_5}
M^\prime(0) + C_3 \leq \left( M^\prime(t_1) + C_3 \right) e^{\frac{t_1}{C_2}}, \quad t_1 \geq 0.
\end{equation}

If there exists $t_0 \geq 0$ such that $M^\prime(t_0) - C_3 > 0$, then (\ref{eq_proof_lem_Mprime_4}) implies that $M^\prime(t) \rightarrow + \infty$ as $t \rightarrow + \infty$, a contradiction with Theorem \ref{thm_main_damped}-\emph{(ii)}. Hence, one has $M^\prime(t) \leq C_3$, for all $t \geq 0$. This gives the upper bound of (\ref{eq_lem_Mprime_estimate_1}). The lower bound of (\ref{eq_lem_Mprime_estimate_1}) is given by (\ref{eq_proof_lem_Mprime_5}).
\end{proof}

Now, we prove Theorem \ref{thm_main_damped}-\emph{(iii)}. Consider $\left( u^0, u^1 \right) \in H_0^1(\Omega) \times L^2(\Omega)$ such that the solution $u$ of \textnormal{(\ref{KG_nL})} exists on $\mathbb{R}_+$. Set 
\[E_\mathscr{L}(t) = \frac{1}{2} \left\Vert u(t) \right\Vert_{H_0^1}^2 + \frac{1}{2} \left\Vert \partial_t u(t) \right\Vert_{L^2}^2 = E\left( u(t), \partial_t u(t) \right) + \int_\Omega F(u(t)) \mathrm{d}x, \quad t \geq 0.\]
By (\ref{eq_proof_lem_Mprime_1}) and (\ref{eq_proof_lem_Mprime_2}), one has
\begin{equation}\label{proof_thm_main_damped_iii_1}
E_\mathscr{L}(t) \leq C_0 M^{\prime \prime}(t) + C_1, \quad t \geq 0,
\end{equation}
for some $C_0 = C_0(f) > 0$, and $C_1 > 0$ of the form $\tilde{C}_0 + \tilde{C}_1 \left\vert E\left( u^0, u^1 \right) \right\vert$, with $\tilde{C}_0 = \tilde{C}_0(f, \gamma) > 0$ and $\tilde{C}_1 = \tilde{C}_1(f, \gamma) > 0$. 

\begin{rem}
In particular, without the restrictive assumption of Theorem \ref{thm_main_damped}-\emph{(iii)}, Lemma \ref{lem_Mprime_estimate} and (\ref{proof_thm_main_damped_iii_1}) imply
\[\int_0^T E_\mathscr{L}(t) \mathrm{d}t \lesssim 1 + T, \quad T \geq 0,\]
with a constant depending on $f$, $\gamma$, and $\left( u^0, u^1 \right)$. This is an average estimate of $E_\mathscr{L}$. It won't be used later on : to derive a non-average estimate, it seems necessary to estimate the derivative of $E_\mathscr{L}$, which is undefined without the additional assumption of Theorem \ref{thm_main_damped}-\emph{(iii)}.
\end{rem}

For the rest of the proof, we write $C$ and $C^\prime$ for some positive constants, that may change from line to line, and which depend on $\Omega$, $f$ and $\gamma$. We assume that $f$ satisfies \textnormal{(\ref{hyp_f_1})} with $p \leq \frac{d}{d - 2}$. Together with the Sobolev embedding $H^1(\Omega) \hookrightarrow L^{\frac{2d}{d - 2}}$, it gives $f(u(t)) \in L^2(\Omega)$, implying
\begin{equation}\label{proof_thm_main_damped_iii_2}
\int_\Omega \left\vert \partial_t u(t) f(u(t)) \right\vert \mathrm{d}x \leq C \left\Vert \partial_t u(t) \right\Vert_{L^2} \left( \left\Vert u(t) \right\Vert_{L^2} + \left\Vert u(t) \right\Vert_{H_0^1}^p \right), \quad t \geq 0.
\end{equation}
In particular, the derivative of $E_\mathscr{L}$ is well-defined, and using also (\ref{eq_energy_equality}), one has
\[E_\mathscr{L}^\prime(t) = - \int_\Omega \gamma \left\vert \partial_t u \right\vert^2 \mathrm{d}x + \int_\Omega \partial_t u(t) f(u(t)) \mathrm{d}x \leq \int_\Omega \partial_t u(t) f(u(t)) \mathrm{d}x, \quad t \geq 0.\]
Together with (\ref{eq_poincare}) and (\ref{proof_thm_main_damped_iii_2}), one finds
\[E_\mathscr{L}^\prime(t) \leq C \left\Vert \partial_t u(t) \right\Vert_{L^2} \left\Vert u(t) \right\Vert_{H_0^1} \left( 1 + \left\Vert u(t) \right\Vert_{H_0^1}^{p - 1} \right), \quad t \geq 0,\]
Using a second time the fact that $p \leq \frac{d}{d - 2}$, and $d \geq 3$, one has $p - 1 \leq 2$, implying
\[E_\mathscr{L}^\prime(t) \leq C E_\mathscr{L}(t) \left( 1 + E_\mathscr{L}(t) \right), \quad t \geq 0.\]
Integrating, this gives
\begin{equation}\label{proof_thm_main_damped_iii_3}
E_\mathscr{L}(t_1) \leq E_\mathscr{L}(t_0) \exp \left( C \int_{t_0}^{t_1} \left( 1 + E_\mathscr{L}(s) \right) \mathrm{d}s \right), \quad t_1 \geq t_0 \geq 0.
\end{equation}

\paragraph{Estimate for $t \leq 1$.} Choosing $t_0 = 0$ in (\ref{proof_thm_main_damped_iii_3}) and using (\ref{proof_thm_main_damped_iii_1}), one finds
\[E_\mathscr{L}(t_1) \leq E_\mathscr{L}(0) \exp \left( \int_0^1 \left( C + C^\prime M^{\prime \prime}(s) \right) \mathrm{d}s \right), \quad 1 \geq t_1 \geq 0.\]
Note that $\left\vert E\left( u^0, u^1 \right) \right\vert \lesssim E_\mathscr{L}(0)$, with a constant depending on $f$, and that $\left\vert M^\prime(0) \right\vert \lesssim E_\mathscr{L}(0)$. Hence, using Lemma \ref{lem_Mprime_estimate}, one obtains
\[E_\mathscr{L}(t_1) \leq E_\mathscr{L}(0) \exp \left( C + C^\prime E_\mathscr{L}(0) \right), \quad 1 \geq t_1 \geq 0,\]
We choose the function $\alpha$ of Theorem \ref{thm_main_damped}-\emph{(iii)} so that $\alpha(s) \geq 2 s \exp \left( C + 2 C^\prime s \right)$, for all $s \in [0, 1]$.

\paragraph{Estimate for $t \geq 1$.} Integrating (\ref{proof_thm_main_damped_iii_3}), and using (\ref{proof_thm_main_damped_iii_1}), one finds
\begin{equation}\label{proof_thm_main_damped_iii_4}
E_\mathscr{L}(t_1) \leq \left( \int_{t_1 - 1}^{t_1} \left( C + C^\prime M^{\prime \prime}(t) \right) \mathrm{d}t \right) \exp \left( \int_{t_1 - 1}^{t_1} \left( C + C^\prime M^{\prime \prime}(s) \right) \mathrm{d}s \right), \quad t_1 \geq 1.
\end{equation}
Lemma \ref{lem_Mprime_estimate} gives
\begin{equation}\label{proof_thm_main_damped_iii_5}
\int_{t_1 - 1}^{t_1} M^{\prime \prime}(t) \mathrm{d}t \leq 2 \left( c_0 + c_1 \left\vert E\left( u^0, u^1 \right) \right\vert \right) + \left\vert M^\prime(0) \right\vert e^{c_2 ( 1 - t_1 )}, \quad t_1 \geq 1,
\end{equation}
where $c_0$, $c_1$ and $c_2$ are the constants of Lemma \ref{lem_Mprime_estimate}. 

Using (\ref{proof_thm_main_damped_iii_4}) and (\ref{proof_thm_main_damped_iii_5}), we can complete the proof of Theorem \ref{thm_main_damped}-\emph{(iii)}. We first prove (\ref{eq_thm_H1_estimate_1}). By (\ref{proof_thm_main_damped_iii_5}), one has
\[\int_{t_1 - 1}^{t_1} M^{\prime \prime}(t) \mathrm{d}t \leq C + C^\prime E_\mathscr{L}(0), \quad t_1 \geq 1.\]
Hence, (\ref{proof_thm_main_damped_iii_4}) gives 
\[E_\mathscr{L}(t_1) \leq \left( C + C^\prime E_\mathscr{L}(0) \right) \exp \left( C + C^\prime E_\mathscr{L}(0) \right), \quad t_1 \geq 1,\]
We choose the function $\alpha$ of Theorem \ref{thm_main_damped}-\emph{(iii)} so that $\alpha(s) \geq \left( C + 2 C^\prime s \right) \exp \left( C + 2 C^\prime s \right)$, for all $s \geq 1$. This completes the proof of (\ref{eq_thm_H1_estimate_1}).

Finally, we prove (\ref{eq_thm_H1_estimate_2}). There exists $T > 0$, which depends on $f$, $\gamma$, and $E_\mathscr{L}(0)$, such that 
\[\left\vert M^\prime(0) \right\vert e^{c_2 ( 1 - t_1 )} \leq c_0, \quad t_1 \geq T.\]
Hence, (\ref{proof_thm_main_damped_iii_4}) and (\ref{proof_thm_main_damped_iii_5}) imply
\[E_\mathscr{L}(t_1) \leq \left( C + C^\prime \left\vert E\left( u^0, u^1 \right) \right\vert \right) \exp \left( C + C^\prime \left\vert E\left( u^0, u^1 \right) \right\vert \right), \quad t_1 \geq T.\]
Up to increasing the constant $c$ defining the function $\alpha$, this completes the proof of (\ref{eq_thm_H1_estimate_2}).

\appendix

\section{Proof of the elementary lemmas}

\subsection{Proof of Lemma \ref{lem_basic_explosion}}

Assume by contradiction that there exists $T_0 \geq T$ such that $M^\prime(T_0) > 0$. In particular, one has $M(T_0) > 0$. We claim that $M(t) > 0$ for all $t \geq T_0$. Indeed, assume by contradiction that there exists $T_1 > T_0$ such that $M(T_1) = 0$ and $M(t) > 0$ for $t \in [T_0, T_1)$. Then, one has $M^{\prime \prime}(t) \geq 0$ for $t \in [T_0, T_1)$, implying that $M^\prime(t) \geq M^\prime(T_0) > 0$ for $t \in [T_0, T_1)$. In particular, this gives $0 = M(T_1) \geq M(T_0) > 0$, a contradiction. 

As $M(t) > 0$ for all $t \geq T_0$, one has $M^{\prime \prime}(t) > 0$ for $t \geq T_0$, implying $M^\prime(t) \geq M^\prime(T_0)$, for $t \geq T_0$. In particular, $M(t) \rightarrow + \infty$ as $t \rightarrow + \infty$. Integrating two times the inequality
\[\frac{M^\prime(t)}{M(t)} \leq \delta \frac{M^{\prime \prime}(t)}{M^\prime(t)}, \quad t \geq T_0,\]
one obtains
\[\left( t - T_0 \right) \frac{M^\prime(T_0)}{M(T_0)^{\frac{1}{\delta}}} \leq \left( \frac{1}{\delta} - 1 \right)^{-1} \left( \frac{1}{M(T_0)^{\frac{1}{\delta} - 1}} - \frac{1}{M(t)^{\frac{1}{\delta} - 1}} \right), \quad t \geq T_0.\]
For $t$ sufficiently large, this gives a contradiction.

\subsection{Proof of Lemma \ref{lem_basic_exponential_estimate}}

Set $m(t) = M_0(t) e^{\sqrt{C} t}$, $t \geq 0$. One has $m^{\prime \prime} \geq 2 \sqrt{C} m^\prime$, implying
\[m(t_1) \geq m(t_0) + \frac{m^\prime(t_0)}{2 \sqrt{C}} \left( e^{2 \sqrt{C} ( t_1 - t_0 )} - 1 \right), \quad t_1 \geq t_0 \geq 0.\]
This gives
\[M_0(t_1) \geq m(t_0) e^{- \sqrt{C} t_1} + \frac{m^\prime(t_0)}{2 \sqrt{C}} \left( e^{\sqrt{C} ( t_1 - 2 t_0 )} - e^{- \sqrt{C} t_1} \right), \quad t_1 \geq t_0 \geq 0.\]
If there exists $t_0 \geq 0$ such that $m^\prime(t_0) > 0$, then $M_0(t) \rightarrow + \infty$. Conversely, if $m^\prime(t) \leq 0$ for all $t \geq 0$, then one has $m(t) \leq m(0)$ for $t \geq 0$, that is, $M_0(t) \leq M_0(0)e^{- \sqrt{C} t }$, for $t \geq 0$.

\printbibliography

\noindent
\textsc{Perrin Thomas:} \texttt{perrin@math.univ-paris13.fr}

\noindent
\textit{Laboratoire Analyse Géométrie et Application, Institut Galilée - UMR 7539, CNRS/Université Sorbonne Paris Nord, 99 avenue J.B. Clément, 93430 Villetaneuse, France}

\end{document}